\documentclass[11pt,oneside]{amsart}

\usepackage{lmodern}    
\usepackage{mathpazo}
\usepackage{amsmath,amssymb,amsthm}
\usepackage{graphicx}
\usepackage[margin=1in]{geometry}
\usepackage[colorlinks=true]{hyperref}
        \hypersetup{urlcolor=RoyalBlue,linkcolor=RoyalBlue,citecolor=ForestGreen}
\usepackage[nameinlink]{cleveref}

\usepackage[usenames,svgnames,dvipsnames]{xcolor}

\theoremstyle{plain}
\newenvironment{subproof}[1][\proofname]{%
  \begin{proof}[#1]%
}{%
  \end{proof}%
}

\usepackage{enumitem}
\setenumerate[0]{label=(\alph*)}

\newcommand\ol\overline
\newcommand\dang\measuredangle

\newtheorem{theorem}{Theorem}[section]
\newtheorem{lemma}[theorem]{Lemma}

\newtheorem{fact}[theorem]{Fact}

\theoremstyle{definition}

\theoremstyle{remark}

\numberwithin{equation}{section}

\begin{document}

\title{Extreme Points on Circumconics Induced by Isogonal Conjugates in a Triangle}

\author{Daniel Hu}
\address{Los Altos High School, Los Altos, CA}
\email{daniel.b.hu@gmail.com}

\maketitle

\begin{abstract}
    We first introduce a configuration of arbitrary isogonal conjugates related to a known property concerning the spiral center of two pairs of isogonal conjugates. We then consider a special case where two conics are tangent at exactly two points. Finally, we apply the discoveries made in both configurations to state a general result concerning the extreme points (those lying on either the major or minor axis) of certain circumconics of a triangle.
\end{abstract}

\section{Introducing the Configuration}

\subsection{Conventions}

The reader should be familiar with the relationship between isogonal conjugation and circumconics of a triangle, as well as the characterization of a conic by cross ratio properties.

All angles used here are directed angles mod $180^\circ$. We use $(XYZ)$ to denote the circumcircle of three points $X, Y, Z$. We use $(VWXYZ)$ to denote the circumconic of $V, W, X, Y, Z$. The symbol $\infty_{XY}$ denotes the point of infinity along line $XY$. The expression $\mathcal C(AB; CD)$ for conic $\mathcal C$ denotes the cross ratio $(AB; CD)$ on $\mathcal C$.

This paper exclusively employs synthetic and projective techniques in order to provide a purely geometric perspective on the configuration.

\subsection{The Configuration}

The first configuration is as follows. In $ABC$ with circumcircle $\Omega$, let:
    \begin{itemize}
        \item $P$ and $P'$ be isogonal conjugates,
        \item $Q$ lie on line $PP'$,
        \item $\Omega$ meet $(ABCPP')$ at $D \ne A, B, C$,
        \item and $DP$ meet $\Omega$ at $X \ne D$.
    \end{itemize}

The following is the main result of our paper:

\begin{theorem}\label{original-statement}
    Given the aforementioned configuration, we have that:
    \begin{enumerate}
        \item 
        $(PQX), (ABCPQ)$ are tangent at $P$.
        \item 
        if $(PQX)$ and $(ABCPQ)$ meet at $L \ne P, Q$, then $PL$ and $(ABCPX)$ intersect on the radical axis of $\Omega$ and $(PQX)$.
    \end{enumerate}
\end{theorem}

\begin{figure}[!htbp]\centering
    \includegraphics[width=200pt]{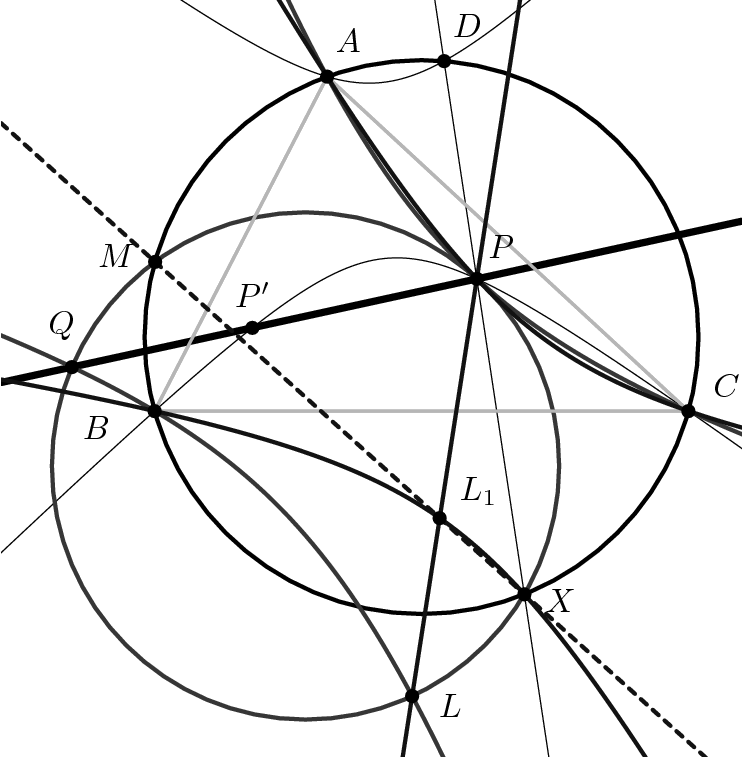}
    \caption{Main Configuration 1}
    \label{fig:1}
\end{figure}

From this main result, we are able to deduce the following two results as well:

\begin{theorem}\label{second-main}
    For $ABC$ with circumcircle $\Omega$ and $P$ with isogonal conjugate $P'$, let $D = \Omega \cap (ABCPP')$. Let $X = DP \cap \Omega, Y = DP' \cap \Omega$. Let $(PXY)$ meet $PP'$ at $Z$.
    
    (a) $(PXY)$ and $(ABCPZ)$ are tangent at $P$ and $Z$.
    
    (b) Let $(PXY)$ meet $(ABCPX)$ at $Z_1, Z_2$. Let $\Omega$ meet $XZ_1, XZ_2$ at $Z_1', Z_2'$. Then $Z_1'Z_2'$ is tangent to $(ABCPP')$ at $P'$.
\end{theorem}

\begin{theorem}
    Let $ABC$ have isogonal conjugates $P, P'$ with $D = (ABC) \cap (ABCPP')$. Let $D' \in (ABC)$ such that $DD' \parallel PP'$. Then the following are equivalent:
    \begin{itemize}
        \item either $PP'$ is tangent to $(ABCD'P)$ or $PP'$ passes through the center of $(ABCD'P)$
        \item $P$ lies on either the major or minor axis of $(ABCD'P)$.
    \end{itemize}
\end{theorem}

\section{Preliminary Lemmas}

To prove this property, we first start with a very simple fact.

\begin{fact}\label{simple-fact}
    Fix conic $\mathcal H$ containing points $A, B, C$ and line $\ell$ containing $D$. Vary point $E$ on $\mathcal H$; let the circumconic $\mathcal E$ of $ABCDE$ meet $\ell$ at $F \ne D$. Then $EF$ passes through a fixed point $G$ on $\mathcal H$.
\end{fact}

\begin{proof}
    Let $EF$ meet $\mathcal H$ at $G \ne E$. Then
    \[\mathcal H(BC; AG) \stackrel{E}{=} \mathcal E(BC; AF) \stackrel{D}{=} (DB, DC; DA, \ell).\]
    Since $(DB, DC; DA, \ell)$ is fixed, by definition of a conic (\cite{ref:Cross}), $G$ is also fixed, as desired.
\end{proof}

We now provide two related lemmas. Labeling is distinct from the original configuration.

\begin{lemma}\label{two-circumconics}
    For $A, B, C, P, P', Q, Q'$ such that $(P, P'), (Q, Q')$ are pairs of isogonal conjugates, let $R, R'$ respectively lie on $PQ, P'Q'$. Then $(ABCPR)$ intersects $(ABCQ'R')$ on $RR'$.
\end{lemma}

\begin{figure}[!htbp]\centering
    \includegraphics[width=200pt]{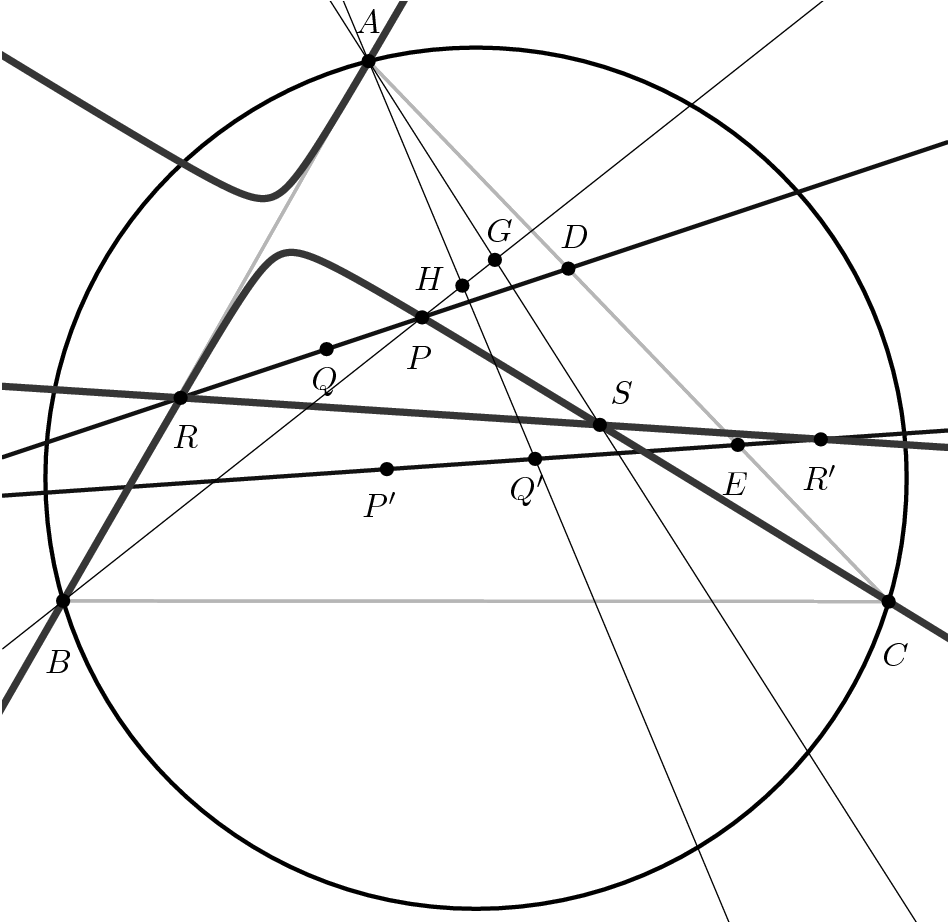}
    \caption{Lemma 2.2}
\end{figure}

\begin{subproof}
    Let $AC$ meet $PQ$ at $D$ and $P'Q'$ at $E$. Let $RR'$ meet $BC$ at $F$ and $(ABCPR)$ at $S \ne R$. Let $AS$ meet $BP$ at $G$. By Pascal's Theorem (\cite{ref:Pascal}) on $ASRPBC$, $D, F, G$ are collinear.
    
    Let $AQ'$ meet $BP$ at $H$; $P'Q'$ meet $BC$ at $I$; and $AP'$ meet $BQ$ at $J$. Let $HJ$ meet $PQ$ at $D_0$ and $P'Q'$ at $I_0$. By the Dual of Desargues' Involution Theorem (DDIT), $(AB, AD_0), (AP, AJ), (AQ, AH)$ are pairs of an involution (\cite{ref:Elementary}, 133), and $(BA, BI_0), (BP', BH), (BJ, BQ')$ are pairs of an involution, so $D_0$ lies on $AC$ and $I_0$ lies on $BC$. Thus $D \equiv D_0$ and $I \equiv I_0$, so $D, H, I$ are collinear.
    
    By Desargues' Theorem (\cite{ref:Desargues}) on triangles $AQ'E$ and $GBF$, $AG, BQ', EF$ concur. By the converse of Pascal on $Q'R'SACB$, $S$ lies on $(ABCQ'R')$ as desired.
\end{subproof}

\begin{lemma}\label{isogonal-miquel}
    Let $P, P'$ and $Q, Q'$ be two pairs of isogonal conjugates in $ABC$. Then the spiral center from $PQ$ to $Q'P'$ lies on $(ABC)$.
\end{lemma}

\begin{figure}[!htbp]\centering
    \includegraphics[width=200pt]{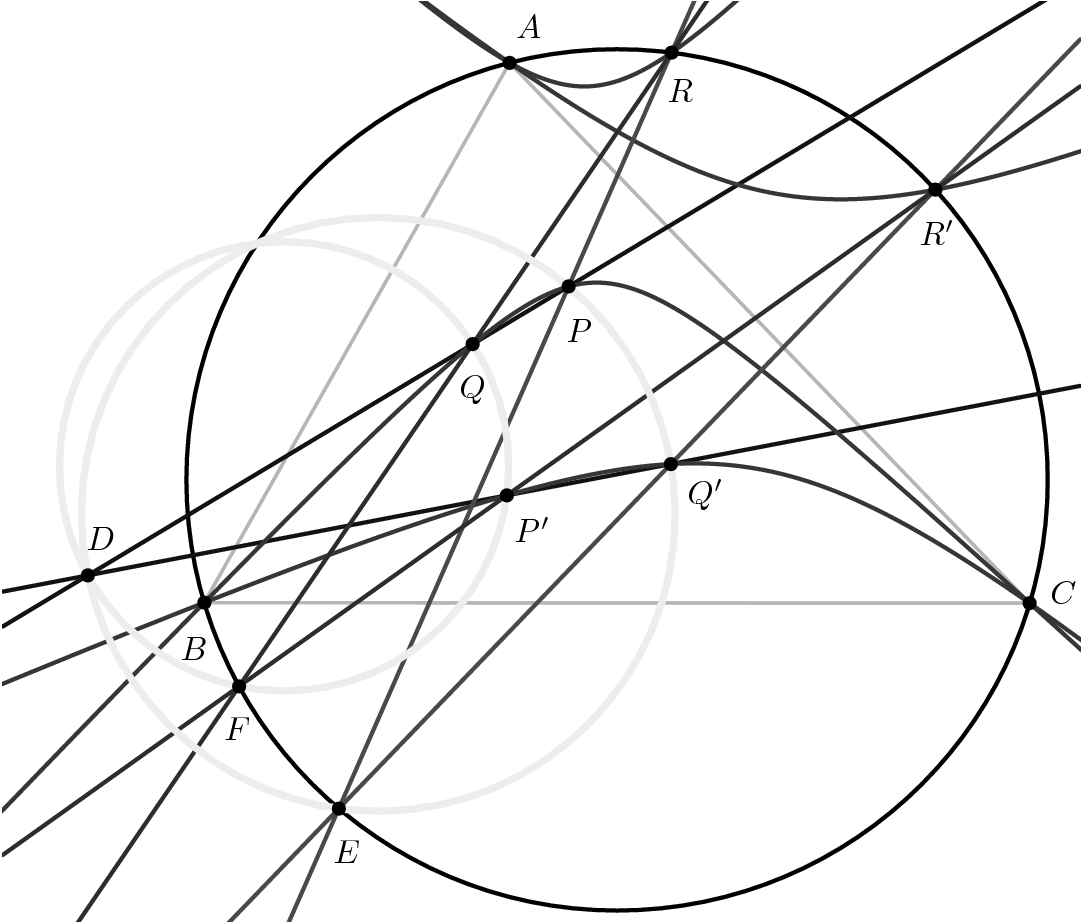}
    \caption{Noncollinear Isogonal Conjugates}
\end{figure}

\begin{subproof}
    Let $S, S'$ lie on $(ABC)$ such that $AS' \parallel PQ$ and $AS \parallel P'Q'$; let Let $R, R'$ lie on $(ABC)$ such that $BC \parallel RS \parallel R'S'$. Then (\cite{ref:Lemmas}, Delta 7.1) $R$ and $R'$ are the isogonal conjugates of the respective points of infinity along $P'Q'$ and $PQ$, and thus respectively lie on $(ABCPQ)$ and $(ABCP'Q')$.
    
    By applying isogonal conjugation on \Cref{two-circumconics}, if $RP$ meets $Q'R'$ at $E$ and $RQ$ meets $P'R'$ at $F$, then $E, F$ lie on $(ABC)$.
    
    Now, we take cases on whether or not three of $P, P', Q, Q'$ are collinear.
    
    \begin{figure}[!htbp]\centering
        \includegraphics[width=200pt]{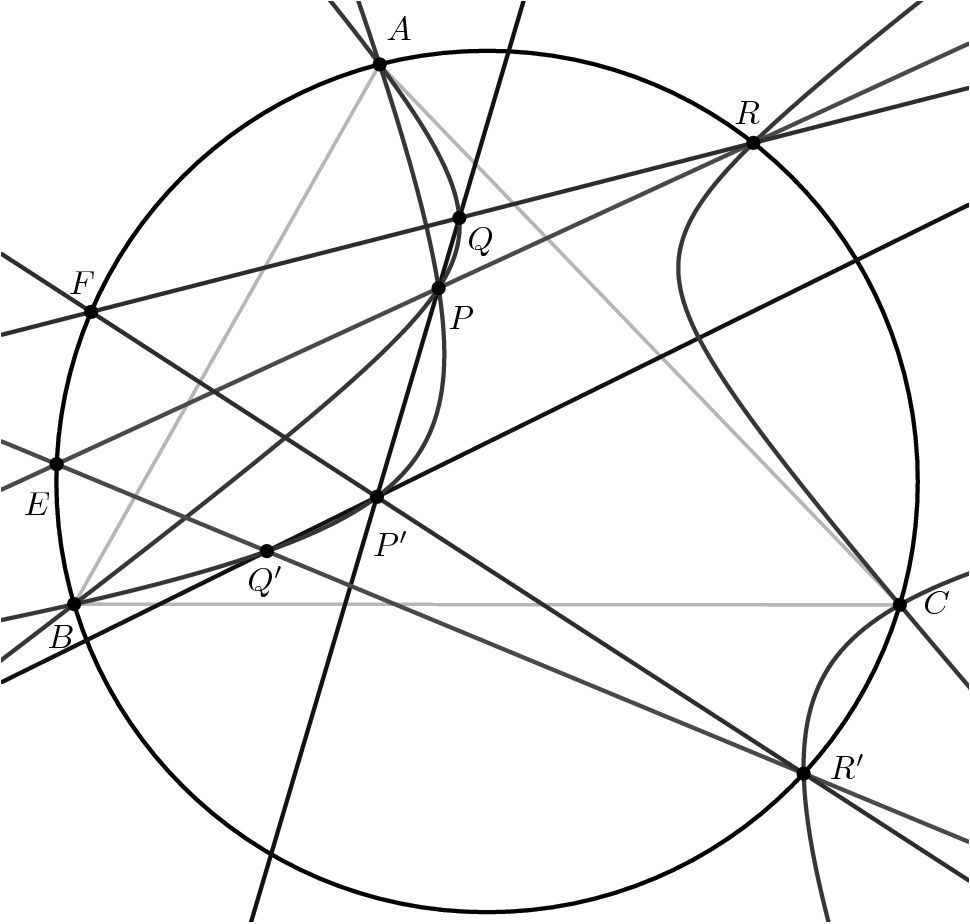}
    \caption{Three Collinear Points}
    \end{figure}
    
    If no three are collinear, then let $D = PQ \cap P'Q'$:
    \[\dang PEQ' = \dang RER' = \dang S'AS = \dang PDQ',\]
    so $E$ lies on $(DPQ')$. Similarly, $F$ lies on $(DP'Q)$. Let $(DPQ')$ meet $(DP'Q)$ at $X \ne D$; then
    \[\dang EXF = \dang EXD + \dang DXF = \dang EPD + \dang DQF = \dang RPQ + \dang PQR = \dang PRQ = \dang ERQ,\]
    so by (\cite{ref:EGMO}, Lemma 10.1), $X$ lies on $(ABC)$ as desired.
    
    Now, if three of the four points are collinear, then WLOG assume $P, P', Q$ are collinear; then $Q' \in (ABCPP')$, so $Q$ cannot lie on $PP'$. Thus
    \[\dang PEQ' = \dang RER' = \dang S'AS = \dang PP'Q',\]
    implying that $PP'Q'E$ is cyclic. We also have
    \[\dang QFP' = \dang RFR' = \dang S'AS = \dang QP'Q',\]
    implying that $P'Q'$ is tangent to $(P'QF)$. Let $(PP'Q'E)$ meet $(P'QF)$ at $X$; then
    \[\dang XPQ = \dang XQ'P', \quad XQP = \dang XP'Q',\]
    implying that $X$ is the desired spiral center. Finally,
    \[\dang EXF = \dang EXP' + \dang P'XF = \dang EPP' + \dang P'QF = \dang RPQ + \dang PQR = \dang ERF\]
    as desired.
\end{subproof}

We may now proceed to prove the main theorem.

\section{The Main Proof}

\begin{theorem}[Part (a)]\label{original-part-1}
    In $ABC$ with circumcircle $\Omega$, let $P, P'$ be isogonal conjugates. Let $Q$ lie on $PP'$. Let $\Omega$ meet $(ABCPP')$ at $D \ne A, B, C$ and $DP$ at $X \ne D$. Then $(PQX), (ABCPQ)$ are tangent at $P$.
\end{theorem}

\begin{figure}[!htbp]\centering
    \includegraphics[width=200pt]{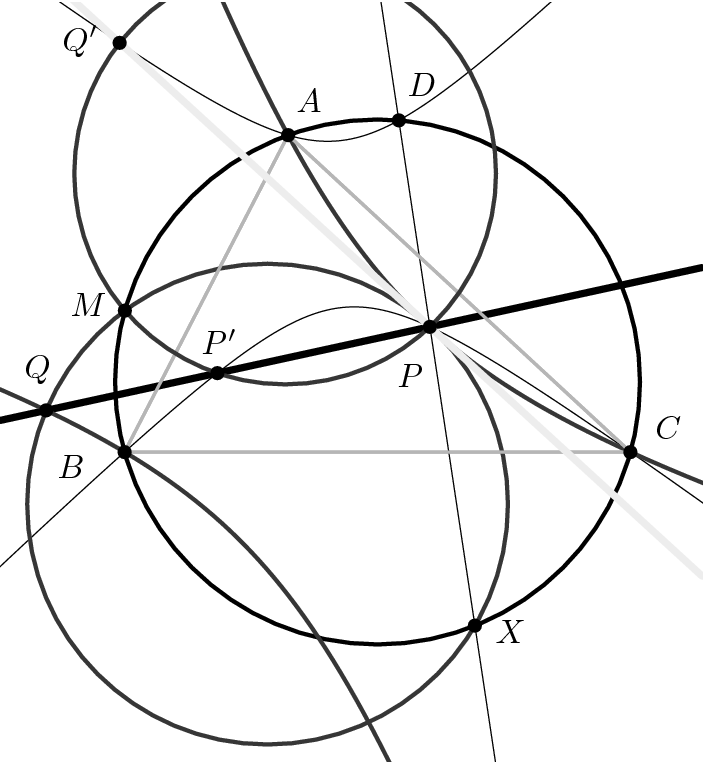}
    \caption{Proving part (a)}
\end{figure}

\begin{proof}
    
    Denote $(ABC)$ by $\Omega$. Let $Q$ have isogonal conjugate $Q'$. Let $M$ be the spiral center from $PQ$ to $Q'P'$; by \Cref{isogonal-miquel} $M \in \Omega$. By definition of spiral center, $PQ'$ is tangent to $(PQM)$, and $(PP'Q'M)$ is cyclic. Let $(PQM)$ meet $\Omega$ at $X' \ne M$. By properties of \Cref{isogonal-miquel}, $PX'$ passes through the intersection of $\Omega$ and $(ABCPP')$, which is precisely $D$, proving that $X \equiv X'$.
    
    This implies that $PQ'$ is tangent to $(PQX)$, so it suffices to prove that $PQ'$ is tangent to $(ABCPQ)$. To prove this, let $\mathcal H$ denote the circumconic of $ABCPP'Q'$ and $\mathcal E$ the circumconic of $ABCPQ$; then
    
    \begin{align*}
        \mathcal E(BC; Q, PQ' \cap \mathcal E) &\stackrel{P}{=} \mathcal H(BC; P'Q') \\
        &= (AB, AC; AP', AQ') \\
        &= (AC, AB; AP, AQ) \\
        &= (AB, AC; AQ, AP) \\
        &= \mathcal E(BC; QP).
    \end{align*}
    
    This is enough to imply that $PQ'$ is tangent to $\mathcal{E}$, proving (a).
\end{proof}

\begin{figure}[!htbp]\centering
    \includegraphics[width=200pt]{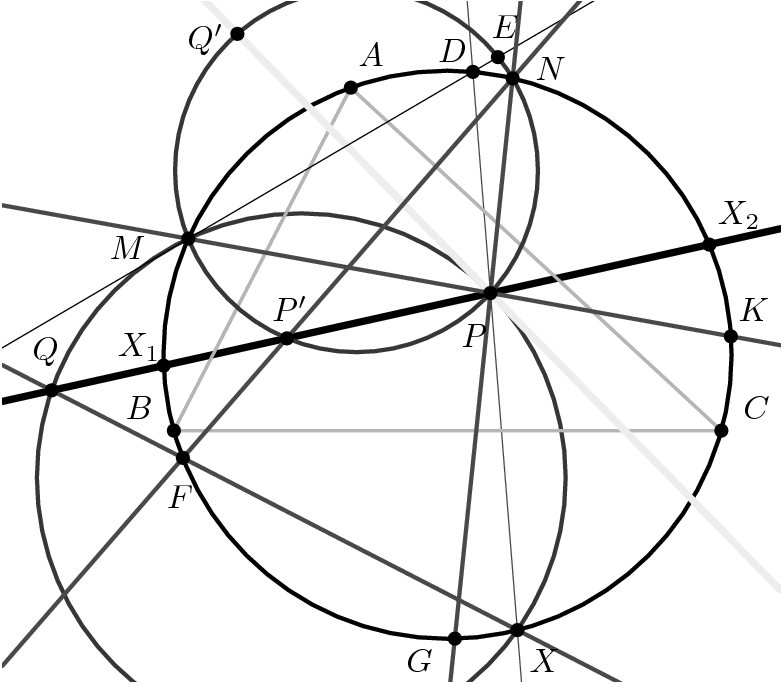}
    \caption{Proving part (b)}
\end{figure}

\begin{theorem}[Part (b)]\label{original-part-2}
    In $ABC$ with circumcircle $\Omega$, let $P, P'$ be isogonal conjugates. Let $Q$ lie on $PP'$. Let $\Omega$ meet $(ABCPP')$ at $D \ne A, B, C$. Let $DP$ meet $\Omega$ at $X$. Let $(PQX), (ABCPQ)$ meet at $L \ne P, Q$. Then $PL$ and $(ABCPX)$ intersect on the radical axis of $\Omega$ and $(PQX)$.
\end{theorem}

\begin{proof}
    
    Let $(PP'Q')$ meet $\Omega$ at $N \ne M$ and $MD$ at $E$. By properties of \Cref{isogonal-miquel}, $NQ'$ passes through the intersection of $(ABC)$ and $(ABCPP')$, which is precisely $D$. Let $\Omega$ meet $NP$, $NP'$, $XQ$, $MP$ at $G, F, F', K$ respectively. By Reim's Theorem, $PP'$ is parallel to both $FK$ and $F'K$, implying that $F \equiv F'$, i.e. $QX, NP'$ concur at a point $F$ on $\Omega$.
    
    Let $PP'$ meet $\Omega$ at (possibly complex) points $X_1$ and $X_2$, and $MN$ at $J$. The key claim is that
    \[(NQ, ND) \quad (NP, NP') \quad (NM, N\infty_{PP'}) \quad (NX_1, NX_2)\]
    are pairs of an involution. To prove this, Desargues' Involution Theorem (\cite{ref:Elementary}, 125) on $NDFX$ yields reciprocal pairs
    \[(Q, ND \cap PP') \quad (P, P') \quad (X_1, X_2)\]
    and applying Desargues' Involution Theorem on $NMKF$ yields reciprocal pairs
    \[(J, \infty_{PP'}) \quad (P, P') \quad (X_1, X_2).\]
    Hence the aforementioned four reciprocal pairs indeed comprise a single involution. Let $\Omega$ meet $NQ$ at $R$ and $N\infty_{PP'}$ at $S$; then by projecting the involution from $N$, $MS, DR, FG$ concur at a point $T$. Denote $(ABCPP'), (PP'Q'), (PP'DQ'E)$ by $\mathcal H, \gamma, \mathcal H'$ respectively; then
    
    \begin{figure}[!htbp]\centering
        \includegraphics[width=200pt]{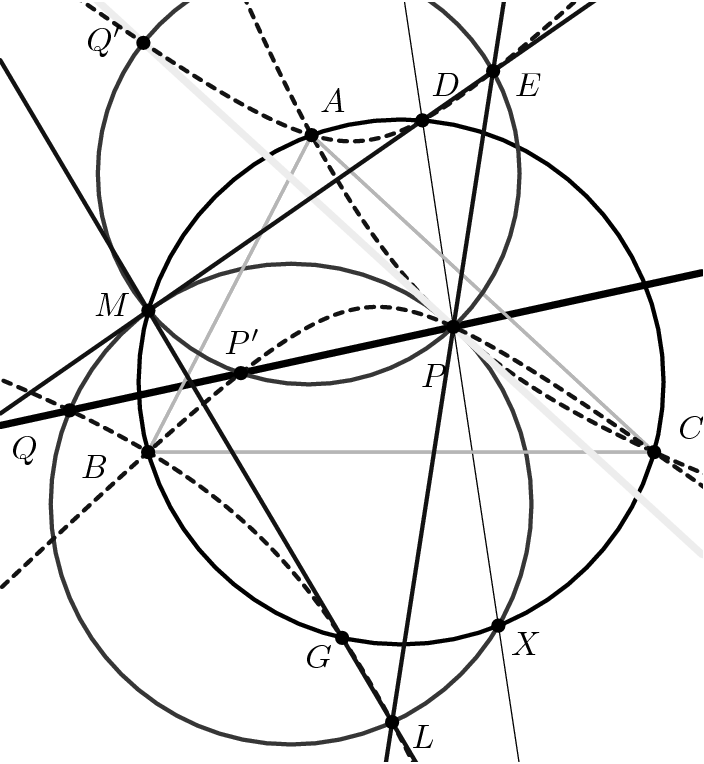}
        \caption{Second diagram for part (b)}
    \end{figure}
    
    \begin{align*}
        \mathcal H(PP'; DQ') &\stackrel{A}{=} (AP', AP; A\infty_{PP'}, AQ) \\
        &\stackrel{A}{=} (P'P; \infty_{PP'}Q) \\
        &\stackrel{N}{=} \Omega(FG; SR) \\
        &\stackrel{T}{=} \Omega(GF; MD) \\
        &\stackrel{N}{=} \gamma(PP'; MQ') \\
        &\stackrel{E}{=} \gamma(EP, EP'; EM, EQ') \\
        &\stackrel{E}{=} \mathcal H'(PP'; DQ')
    \end{align*}
    In other words, $\mathcal H(PP'; DQ') = \mathcal H'(PP'; DQ')$. Thus by definition of a conic (\cite{ref:Cross}), $E \in \mathcal H$.
    
    Let $MG$ meet $(PQM)$ at $H \ne M$. By Reim (\cite{ref:Reim}), both $PH$ and $PE$ are parallel to $DG$, hence $P \in EH$. By properties of \Cref{isogonal-miquel}, $NP$ passes through the intersection of $\Omega$ with $(ABCPQ)$, i.e. $G \in (ABCPQ)$. It is suddenly clear that $H \in (ABCPQ)$ by \Cref{simple-fact}, i.e. $H \equiv L$. Finally, by \Cref{simple-fact} once again with line $PE$ and $\Omega$, the intersections of $(ABCPX)$ with $PE$ and $\Omega$ are collinear with $M$ - i.e., $PE$ meets $MX$ on $(ABCPX)$, completing (b).

\end{proof}

The remainder of our paper will be dedicated to showing each of the aforementioned results.

\section{Circles Tangent to Circumconics at Exactly Two Points}

Having shown the main result in our paper, we will now prove our second main theorem, reiterating \Cref{second-main}.

\begin{theorem}\label{two-tangency-points}
    For $ABC$ with circumcircle $\Omega$ and $P$ with isogonal conjugate $P'$, let $D = \Omega \cap (ABCPP')$. Let $X = DP \cap \Omega, Y = DP' \cap \Omega$. Let $(PXY)$ meet $PP'$ at $Z$.
    
    (a) $(PXY)$ and $(ABCPZ)$ are tangent at $P$ and $Z$.
    
    (b) Let $(PXY)$ meet $(ABCPX)$ at $Z_1, Z_2$. Let $\Omega$ meet $XZ_1, XZ_2$ at $Z_1', Z_2'$. Then $Z_1'Z_2'$ is tangent to $(ABCPP')$ at $P'$.
\end{theorem}

To prove this, we once again start with a simple lemma.

\begin{lemma}\label{parallel-circles}
    Two circles $\Omega, \omega$ meet at a point $X$. Let $A, B, C, D$ lie on $\omega$ such that $AB \parallel CD$. Let $\Omega$ meet $XA, XB, XC, XD$ at $E, F, G, H$. Then $EF \parallel GH$.
\end{lemma}

\begin{subproof}
$\dang EFH = \dang EXH = \dang ACD = \dang CDB = \dang CXB = \dang GHF$.
\end{subproof}

We proceed to prove part (a). Denote $(ABCPP'), (ABCPZ), (ABCPX), (PXY)$ by $\mathcal H, \mathcal E, \mathcal C_1, \gamma$ respectively.

\begin{proof}
    By treating $Z$ as $Q$ as in \Cref{original-statement}, we deduce $\gamma$ and $\mathcal E$ tangent at $P$; it suffices to prove the two tangent at $Z$.
    
    Assume the contrary - that $\gamma$ and $\mathcal E$ meet at a point $Z^* \ne P, Z$. By construction, $Y$ is the spiral center from $PZ$ to $Z'P'$, where $Z'$ is the isogonal conjugate of $Z$. Then by \Cref{original-part-2}, $PZ^*$ and $XY$ meet on $\mathcal C_1$. Let $PZ$ meet $\mathcal C_1$ at $I$. We wish to show $I \in XY$. To prove this,
    \[\mathcal C_1(BC; AI) \stackrel{P}{=} \mathcal H(BC; AP') \stackrel{D}{=} \Omega(BC; AY) \stackrel{X}{=} \mathcal C_1(B, C; A, XY \cap \mathcal C_1),\]
    so $I$ lies on $XY$. Thus $I$ lies on both $PZ$ and $PZ^*$, so we can characterize both $Z$ and $Z^*$ as the unique intersection of $PI$ and $\mathcal E$ other than $P$. This is the desired contradiction.
\end{proof}

A corollary of (a) is that by \Cref{original-part-2}, $YZ$ passes through $\Omega \cap \mathcal E$, which we will denote as $G$. By Reim, we notably have $DG \parallel PZ$.

Before we proceed to part (b), we first state another lemma.

\begin{figure}[!htbp]\centering
    \includegraphics[width=250pt]{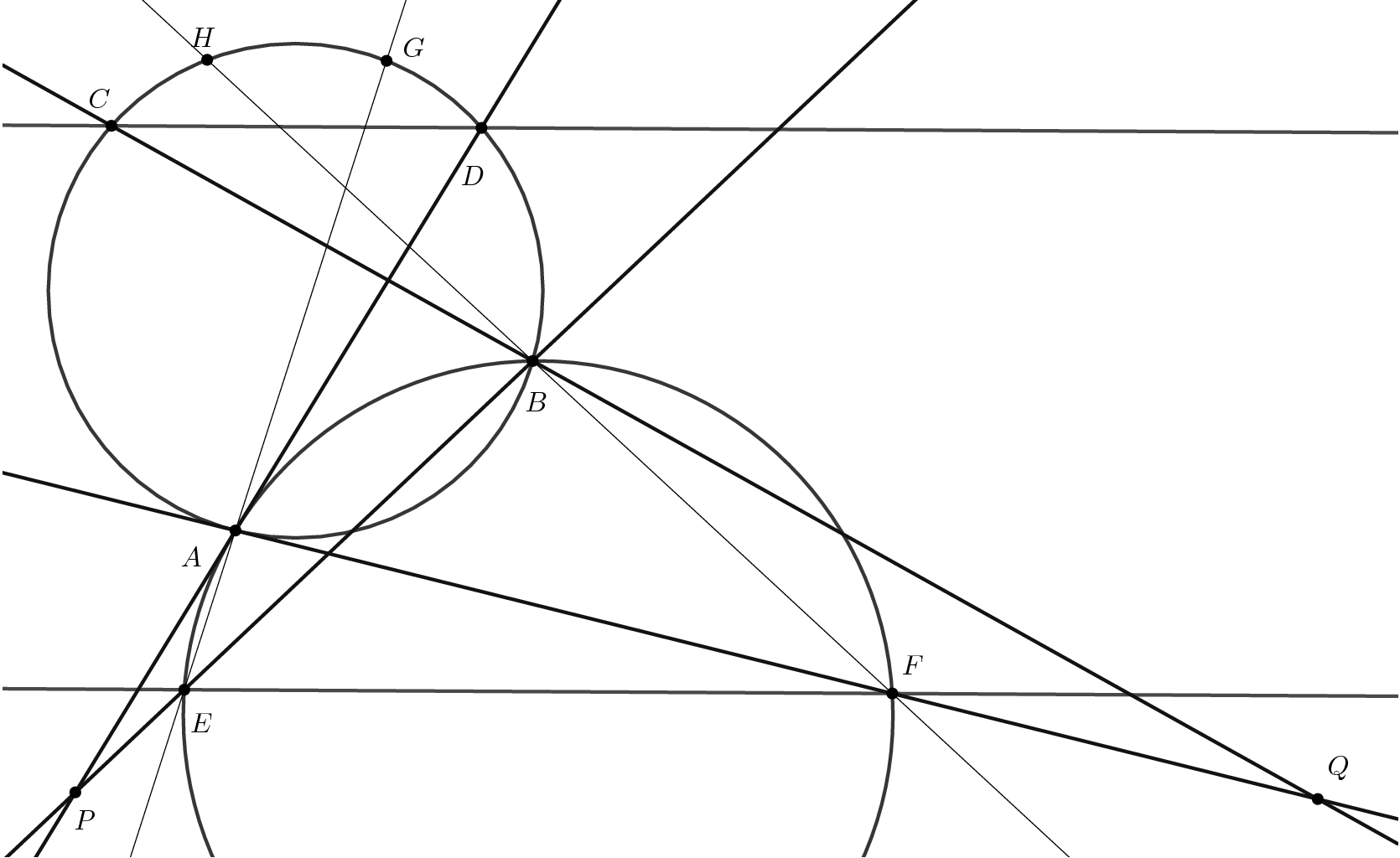}
    \caption{Circles Yielding a Family of Parallel Lines}
\end{figure}

\begin{lemma}\label{circle-conic}
    For conic $\mathcal C$ containing two points $A$ and $B$, let $C, D$ vary on $\mathcal C$ such that $ABCD$ is cyclic. Then $CD$ is parallel to a fixed line.
\end{lemma}

\begin{subproof}
    It suffices to prove that, for two cyclic quadrilaterals $ABCD$ and $ABEF$ with $CD \parallel EF$, then $ABCDEF$ is circumscribed by a single conic. To prove this, let $(ABCD)$ meet $AE, BF$ at $G, H$; $AD$ meet $BE$ at $P$; and $AF$ meet $BC$ at $Q$. Then $GH \parallel EF$ by Reim. From
    \[\dang APB = \dang DAE + \dang AEB = \dang DAG + \dang AFB = \dang HBC + \dang QFB = \dang AQB,\]
    follows $ABPQ$ cyclic, so $PQ \parallel CD$ by Reim. By converse Pascal (\cite{ref:Pascal}) on $CDAFEB$, since $P, Q, CD \cap EF$ are collinear, $ABCDEF$ is circumscribed by a single conic as desired.
\end{subproof}

The above lemma implies the following:

\begin{lemma}\label{vary-conic}
    Let $X_1, X_2$ vary on $\mathcal C_1$ such that $PXX_1X_2$ is cyclic.
    
    (a) $X_1X_2 \parallel PZ$.
    
    (b) If $\Omega$ meets $XX_1, XX_2$ at $X_1', X_2'$, then $P', X_1', X_2'$ are collinear.
\end{lemma}

\begin{figure}[!htbp]\centering
    \includegraphics[width=200pt]{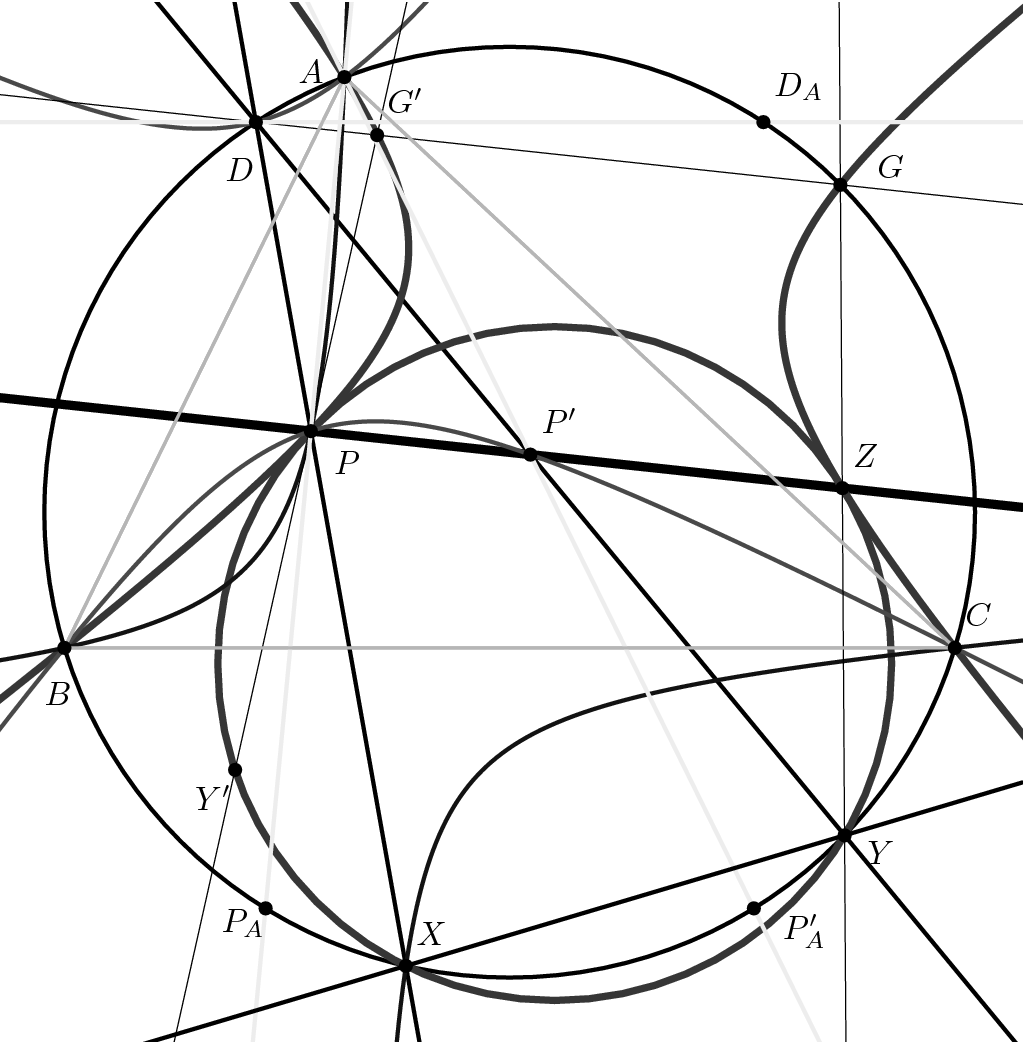}
    \caption{One Tangency Point Becomes Two}
\end{figure}

\begin{subproof}

    (a)
    
    By \Cref{circle-conic}, it suffices to verify this for one choice of $X_1, X_2$.
    
    Let $G' = DG \cap \mathcal E$. Since $\mathcal E$ is tangent to $(XPZ)$ at $P$ and $Z$, both $P$ and $Z$ must be equidistant from the center of $\mathcal E$. Since $GG' \parallel PZ$, $PZGG'$ is an isosceles trapezoid. Let the tangent to $\mathcal C_1$ at $P$ meet $\mathcal E$ at $G^* \ne P$; then
    \[\mathcal E(BC; AG^*) \stackrel{P}{=} \mathcal C_1(BC; AP) \stackrel{X}{=} \Omega(BC; AD) \stackrel{G}{=} \mathcal E(BC; AG'),\]
    implying that $G^* \equiv G'$, i.e. $PG'$ is tangent to $\mathcal C_1$. Let $PG'$ meet $\gamma$ at $Y'$; then $YY' \parallel PZ$, so
    \[\dang Y'PI = \dang PZY = \dang PXY = \dang PXI,\]
    implying that $PY'$ is tangent to $(PXI)$, so $(PXI)$ is tangent to $\mathcal C_1$ at $P$. With this tangency, by choosing $X_1 \equiv P, X_2 \equiv I$, we deduce that all $X_1X_2$ are parallel to $PZ$, as desired.
    
    (b)
    
    All such reciprocal pairs $(X_1, X_2)$ comprise an involution, so projecting from $X$ yields all reciprocal pairs $(X_1', X_2')$ comprising a single involution. Therefore, it suffices to prove that $P'$ lies on $X_1', X_2'$ for two choices of the pair $(X_1, X_2)$. Since $D, P', Y$ are collinear, we already have one pair $(X_1, X_2) \equiv (P, I)$.
    
    For our next pair, let $D_A$ lie on $\Omega$ such that $AD_A \parallel PP'$; let $\Omega$ meet $AP, AP'$ at $AP_A, AP_A'$. Then $BC \parallel DD_A \parallel P_AP_A'$. Let $\mathcal C_1$ meet $AD_A$ at $A_1$ and $XP_A'$ at $A_1'$. Then
    \[\mathcal C_1(BC; PA_1) \stackrel{A}{=} \Omega(BC; P_AD_A) = \Omega(BC; DP_A') \stackrel{X}{=} \mathcal C_1(BC; PA_1'),\]
    so $A_1 \equiv A_1'$, implying that $AD_A, XP_A'$ meet on $\mathcal C_1$. Since $AA_1 \parallel PP'$, by part (a),$(AA_1PX)$ is cyclic. Since $A, P', P_A'$ are collinear, our desired second pair is $(A, P_A')$. This completes this lemma.
    
\end{subproof}

Now we may finally approach part (b) of the main result of this section.

\begin{proof}
    By part (a) of \Cref{vary-conic}, $PZ \parallel Z_1Z_2$; by \Cref{parallel-circles}, this implies that $DF \parallel Z_1'Z_2'$, where $F = XZ \cap \Omega$. By part (b) of \Cref{vary-conic}, $P', Z_1', Z_2'$ are collinear. Therefore, it suffices to show that the $DF$ is parallel to the tangent to $\mathcal H$ at $P'$. Let $U = \Omega \cap PY$; denote $(ABCPU)$ by $\mathcal C_2$. Then
    \[\mathcal C_2(BC; AP) \stackrel{U}{=} (BC; AY) \stackrel{D}{=} \mathcal H(BC; AP') \stackrel{P}{=} \mathcal C_2(B, C; A, PP' \cap \mathcal C_2),\]
    so $PP'$ is tangent to $\mathcal C_2$. Therefore, $U$ is the isogonal conjugate of the point of infinity along the tangent to $(ABCPP')$ at $P'$. It suffices to prove that $AU$ is isogonal to the line through $A$ parallel to $DF$. Since $AD$ is isogonal to the line through $A$ parallel to $PP'$, if we let $W = DF \cap PZ$, it suffices to prove $\dang DWP = \dang DAU$.  This follows from
    \[ \dang DWP = \dang DPW + \dang WDP = \dang XPZ + \dang FDX = \dang XYZ + \dang FYX = \dang FYG = \dang DAU\]
    where the final step follows from $DG \parallel UF \parallel PZ$ by Reim's Theorem, as desired.
\end{proof}

With slight modification, we may phrase this fact as the following:

\begin{theorem}
    In triangle $ABC$ with circumcircle $\Omega$, let $\mathcal C$ be a circumconic of $ABC$ tangent to a circle $\gamma$ at two points $P,Q$. Suppose $PQ$ passes through the isogonal conjugate $P'$ of $P$. Let $P_0$ lie on $\Omega$ such that $P'P_0$ is tangent to $(ABCPP')$. Then one of the intersections $X$ of $\gamma$ and $\Omega$ has the property that $P_0X$ meets $\gamma$ on $(ABCPX)$ at a point distinct from $X$.
\end{theorem}

\begin{proof}
    We will first revisit the labelling in \Cref{original-part-2}. For any point $P$ with isogonal conjugate $P'$ and point $Q$ on $PP'$, there is exactly one circle through $P, Q$ tangent to $(ABCPQ)$. By \Cref{original-part-1}, this circle passes through $X$ (using the labelling of \Cref{original-part-2}).
    
    \begin{figure}[!htbp]\centering
        \includegraphics[width=200pt]{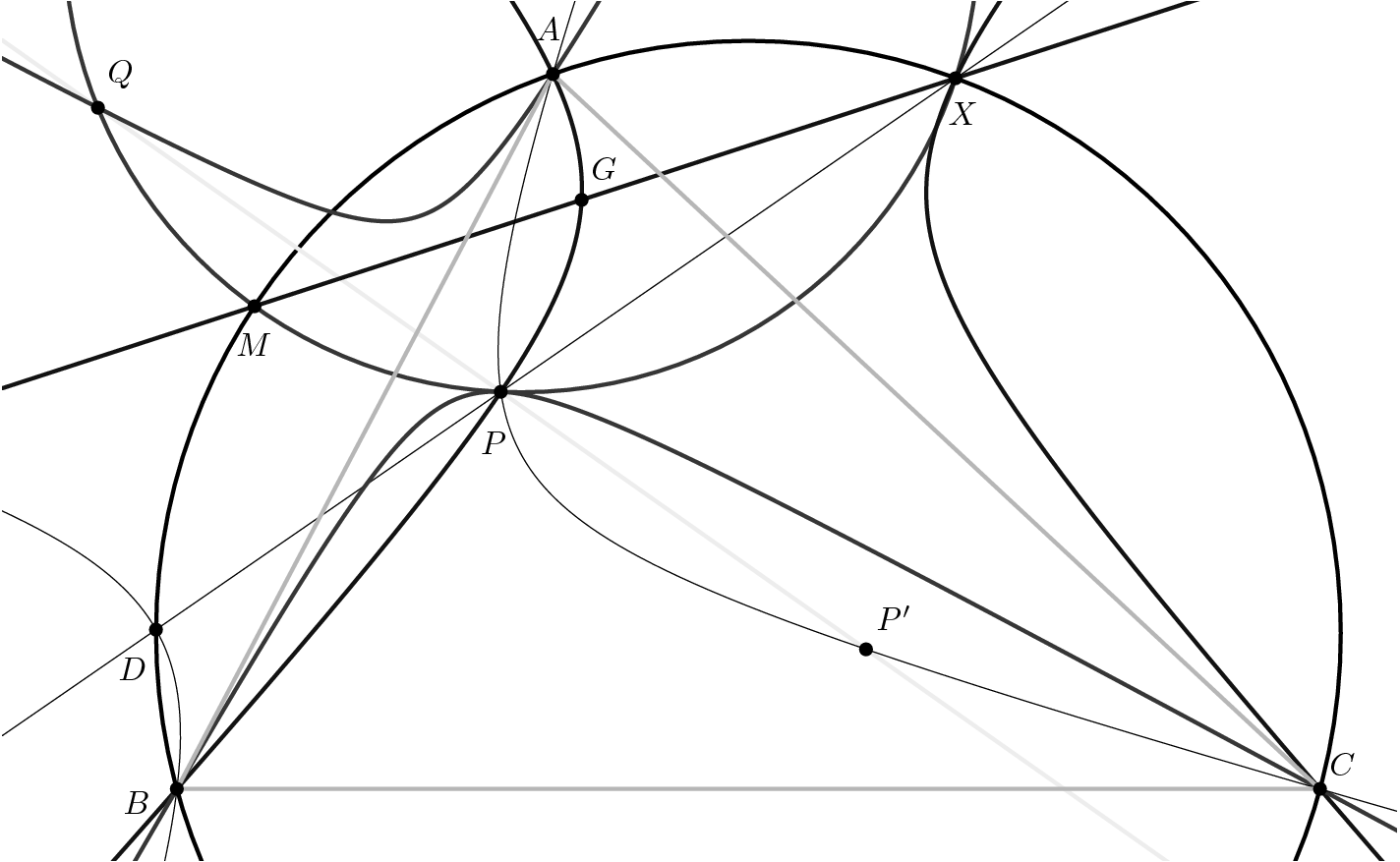}
        \caption{Proving Uniqueness}
    \end{figure}
    
    Now, by \Cref{original-part-2}, we may redefine $G$ to be the intersection of $MX$ and $(ABCPX)$ other than $X$. Now, if $(PQX)$ and $(ABCPQ)$ are not tangent at both $P$ and $Q$, then their other intersection $L$ lies on $PG$. In other words, if $(PQX)$ and $(ABCPQ)$ are tangent at $P$ and $Q$, then $G$ must lie on $PP'$, i.e. $G = PP' \cap (ABCPX)$. Such a point $G$ is unique.
    
    Note that $X$ is fixed. Therefore, given any point $G$ on $(ABCPX)$, we may reconstruct the corresponding $Q$ as follows: we construct $M = XG \cap \Omega$ and $Q = (PMX) \cap PP'$ (if $XG$ is tangent to $\Omega$ then we would set $M \equiv X$, and if $PP'$ is tangent to $(PMX)$ then we would set $Q \equiv P$. Now, when we consider the unique $G$ lying on $PP'$, there is exactly one corresponding $Q$. Therefore, there is exactly one $Q \in PP'$ for which $(ABCPQ)$ and $(PQX)$ are tangent at both $P$ and $Q$. This completes the proof.
\end{proof}

\section{Applications to Extreme Points of Conics}

We shall now show the third of our main results.

\begin{theorem}
    Let $ABC$ have isogonal conjugates $P, P'$ with $D = (ABC) \cap (ABCPP')$. Let $D' \in (ABC)$ such that $DD' \parallel PP'$. Then the following are equivalent:
    \begin{itemize}
        \item either $PP'$ is tangent to $(ABCD'P)$ or $PP'$ passes through the center of $(ABCD'P)$
        \item $P$ lies on either the major or minor axis of $(ABCD'P)$.
    \end{itemize}
\end{theorem}

\begin{proof}
    First, we revisit the configuration in \Cref{two-tangency-points}. We may redefine $\gamma = (PXY)$, $G$ to be the point on $\Omega$ for which $DG \parallel PP'$, and $\mathcal E = (ABCPG)$. Note that the two tangency points $P, Z$ of $\mathcal E$ and $\gamma$ can be treated as two intersections, each of multiplicity two. Now, if $P$ and $Z$ coincide at a single point $P$, then our new definitions of $\mathcal E$ and $\gamma$ would intersect at a single point $P$ with multiplicity 4.
    
    Now, for any given conic containing a given point $P$, there is a circle through $P$ tangent to the conic with multiplicity 4 if and only if $P$ is one of the extreme points of $ABC$. Furthermore, $P$ and $Z$ coincide if and only if $\gamma$ is tangent to $\mathcal E$ with multiplicity 4, which implies that $P$ is one of the extreme points of $\mathcal E$. On the other hand, $P$ and $Z$ coincide if and only if $\gamma$ is tangent to $PP'$ at $P$. By Reim (\cite{ref:Reim}), this occurs if and only if $YP \cap \Omega$ lies on the line through $D$ parallel to $PP'$, i.e. $Y, P, G$ are collinear. This occurs if and only if $DP'$ and $GP$ meet on $\Omega$, which using the labelling in the current configuration, is equivalent to $DP'$ meeting $D'P$ on $\Omega$.
    
    \begin{figure}[!htbp]\centering
        \includegraphics[width=200pt]{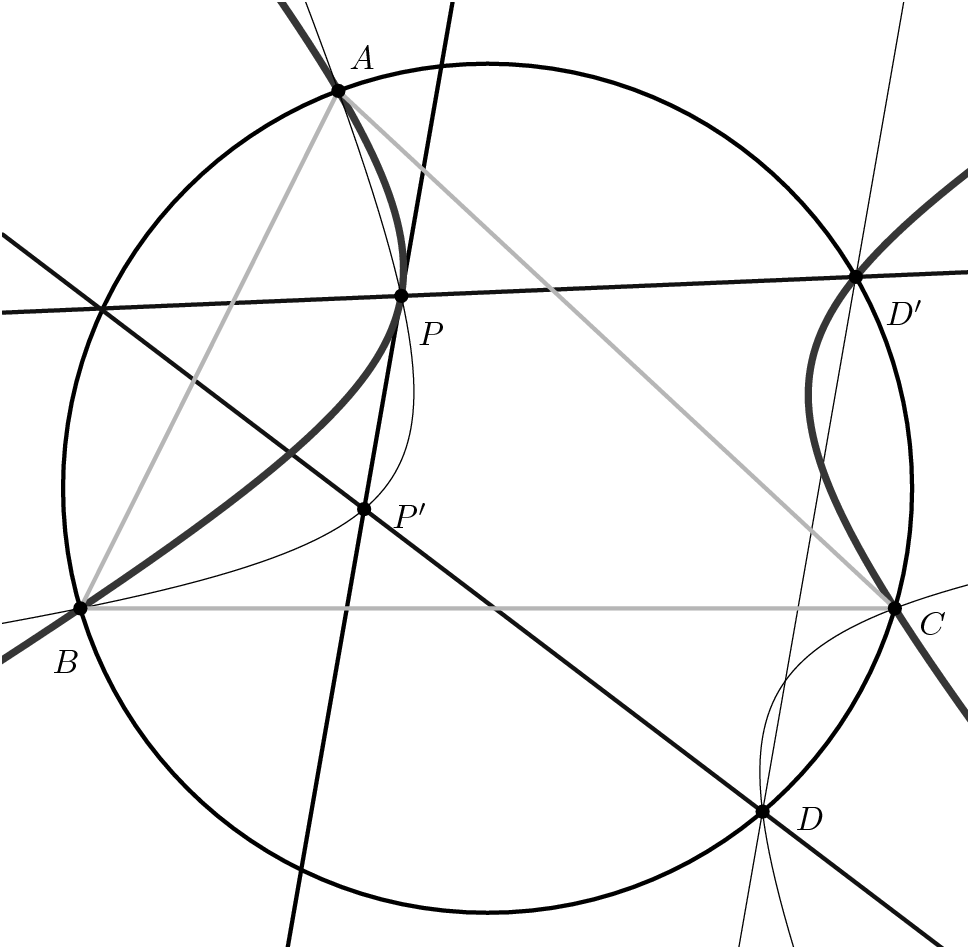}
        \caption{Case 1 - Multiplicity 4}
    \end{figure}

    Now, we claim that $PP'$ is tangent to $(ABCD'P)$ if and only if $DP'$ and $D'P$ intersect on $(ABC)$. Denote $(ABC), (ABCPP'), (ABCD'P)$ by $\Omega, \mathcal H, \mathcal C$ respectively. Then
    \[\mathcal C(BC; AP) \stackrel{D'}{=} \Omega(BC; A, D'P \cap \Omega), \quad \mathcal C(BC; A, PP' \cap \mathcal C) \stackrel{P}{=} \mathcal H(BC; AP') \stackrel{D}{=} \Omega(BC; A, DP' \cap \Omega)\]
    so $PP'$ is indeed tangent to $(ABCD'P)$ if and only if $D'P \cap DP' \in \Omega$. By the above, this occurs if and only if the corresponding $\gamma$ is tangent to $\mathcal C$ with multiplicity 4, which would imply that $P$ is one of the extremes of $\mathcal C$.
    
    Now, if $PP'$ passes through the center of $\mathcal C$, then $P$ cannot be the center of $\mathcal C$ (since $P \in \mathcal C$), so $P$ and the corresponding $Z$ are distinct. Since the two tangency points of $\gamma$ and $\mathcal C$ are collinear with the center of $\mathcal C$, both $P$ and $Z$ are extremes of $\mathcal C$. Conversely, if $P$ and $Z$ are both extremes of $\mathcal C$, then $PP'$ must contain the center of $\mathcal C$, so $PP'$ contains the center of $\mathcal C$ if and only if $P$ and the corresponding $Z$ are two distinct extremes of $\mathcal C$.
    
    \begin{figure}[!htbp]\centering
        \includegraphics[width=200pt]{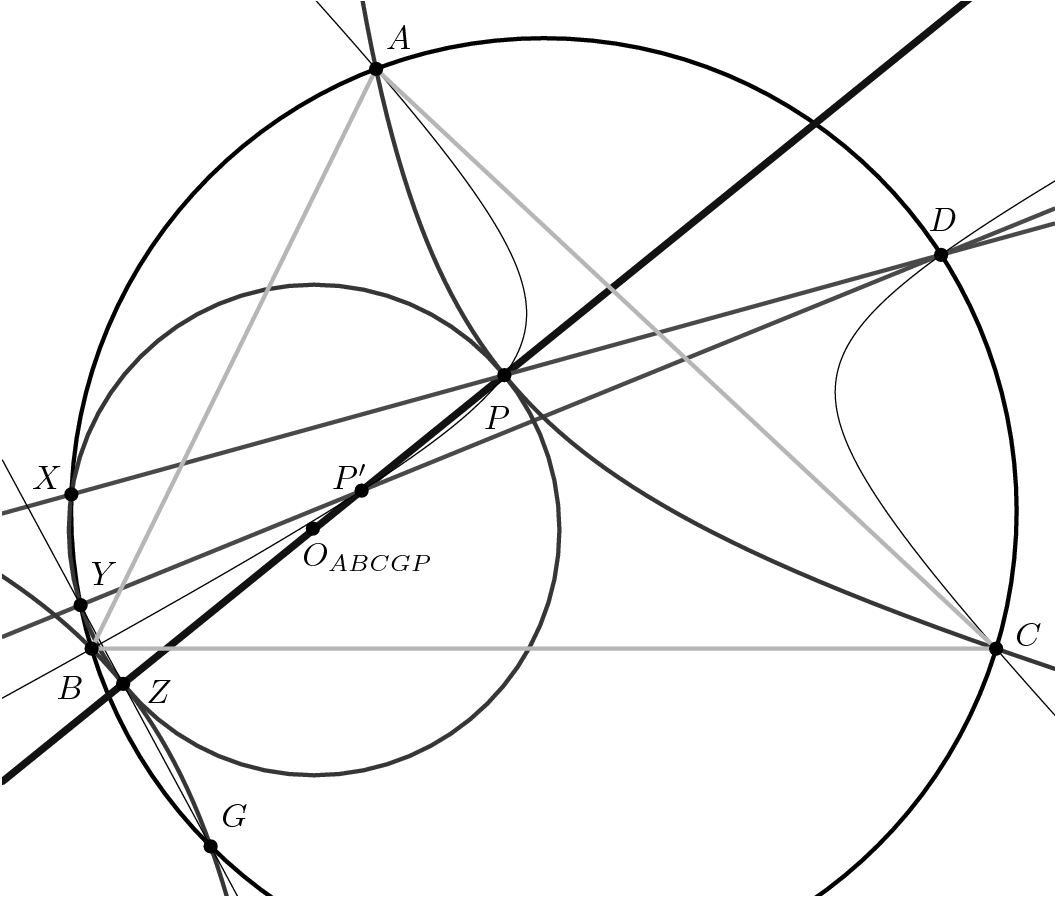}
        \caption{Case 2 - Both Tangency Points are Extreme Points}
    \end{figure}
    
    Finally, we note that if $P$ is one of the extremes of $\mathcal C$, then either the corresponding $\gamma$ to $P$ is tangent to $\mathcal C$ at two extreme points of $\mathcal C$, or the corresponding $\gamma$ is tangent to $\mathcal C$ with multiplicity 4. This completes the if and only if condition, completing the proof.
\end{proof}

\section{Acknowledgements}

The author would like to acknowledge and give special thanks to Michael Diao for helping format the document and diagrams, as well as proofreading and giving significant advice for all parts of the paper.


\begin{thebibliography}{9}
\bibitem{ref:Elementary} Lehmer, Derrick Norman. \textit{Elementary Course in Synthetic Projective Geometry}. General Books, 2010, pp. 88–103.
\bibitem{ref:Lemmas} Andreescu, Titu, et al. \textit{Lemmas in Olympiad Geometry}. XYZ Press, 2016.
\bibitem{ref:EGMO} Chen, Evan. \textit{Euclidean Geometry in Mathematical Olympiads}. Washington, DC: Mathematical Association of America, 2016.
\bibitem{ref:Cross} Birthfield, Stanley. "Conics." Glossary of Terms Journal of Machine Learning, 23 Apr. 1998; \url{robotics.stanford.edu/~birch/projective/node11.html}.
\bibitem{ref:Reim} Bogomolny, Alexander. "Reim's Similar Coins I." Interactive Mathematics Miscellany and Puzzles, \url{www.cut-the-knot.org/m/Geometry/Reim1.shtml}.
\bibitem{ref:Pascal} Weisstein, Eric W. "Pascal's Theorem." Wolfram MathWorld, Wolfram Research, Inc., 14 Sep. 2019, \url{mathworld.wolfram.com/PascalsTheorem.html}.
\bibitem{ref:Desargues} Weisstein, Eric W. "Desargues' Theorem." Wolfram MathWorld, Wolfram Research, Inc., 14 Sep. 2019, \url{mathworld.wolfram.com/DesarguesTheorem.html}.
\end{thebibliography}
\end{document}